\documentclass[11pt,twoside,reqno]{amsart}
\usepackage{amssymb}
\usepackage{enumerate}
\usepackage{amsmath}
\usepackage{a4wide}
\usepackage[usenames]{color}
\usepackage[dvipsnames]{xcolor}
\usepackage{soul}
\usepackage[all]{xy}
\usepackage{array}
\usepackage{tkz-graph}

\usepackage{graphicx}
\usepackage{tabularx}
\usepackage[table]{xcolor}
\usepackage{booktabs}
\usepackage{verbatim}
\usepackage[hidelinks=true]{hyperref} 
\usepackage{listings}

\definecolor{codegreen}{rgb}{0,0.6,0}
\definecolor{codegray}{rgb}{0.5,0.5,0.5}
\definecolor{codepurple}{rgb}{0.58,0,0.82}
\definecolor{backcolour}{rgb}{0.95,0.95,0.92}

\lstdefinestyle{Pstystyle}{
    backgroundcolor=\color{backcolour},   
    commentstyle=\color{codegreen},
    keywordstyle=\color{magenta},
    numberstyle=\tiny\color{codegray},
    stringstyle=\color{codepurple},
    basicstyle=\ttfamily\footnotesize,
    breakatwhitespace=false,         
    breaklines=true,                 
    captionpos=b,                    
    keepspaces=true,                 
    numbers=left,                    
    numbersep=5pt,                  
    showspaces=false,                
    showstringspaces=false,
    showtabs=false,                  
    tabsize=2
}

\usepackage{mathtools}

\usepackage{lmodern}
\usepackage[utf8]{inputenc}
\usepackage[L7x]{fontenc} 
\usepackage{blkarray}
\definecolor{blanchedalmond}{rgb}{1.0, 0.92, 0.8}

\usepackage{array}

\usepackage[
   backend=bibtex, 
   style=numeric,
   isbn=false,
   giveninits=true,
   doi=false,
   url=false]{biblatex}
\newtheorem{theorem}{Theorem}

\newtheorem{lemma}[theorem]{Lemma}

\theoremstyle{remark}

\def\leq{\leqslant} \def\geq{\geqslant}

\graphicspath{ {C:/Users/Zygimantas/Desktop/GRAPHS/} }

\begin{document}

\title{No three algebraic conjugates of degree sixteen sum to zero}

\author{Žygimantas Baronėnas, Paulius Drungilas, and Jonas Jankauskas}

\address{Institute of Mathematics, Faculty of Mathematics and Informatics, Vilnius
University, Naugarduko 24, Vilnius LT-03225, Lithuania}
\email{zygimantas.baronenas@mif.stud.vu.lt}

\address{Institute of Mathematics, Faculty of Mathematics and Informatics, Vilnius
University, Naugarduko 24, Vilnius LT-03225, Lithuania}
\email{paulius.drungilas@mif.vu.lt}

\address{Institute of Mathematics, Faculty of Mathematics and Informatics, Vilnius
University, Naugarduko 24, Vilnius LT-03225, Lithuania}
\email{jonas.jankauskas@mif.vu.lt}


\begin{abstract}
Let $d$ be the smallest positive integer, not divisible by $3$, for which there exists an algebraic number over $\mathbb{Q}$ of degree $d$ whose some three algebraic conjugates sum to zero. Employing the classification of vertex-transitive graphs on 16 vertices of degree 6, we prove that $d\neq 16$. 
This, combined with results obtained by Dubickas, Smyth and Stong \cite{DubickasSmyth2006}, Dubickas and Jankauskas \cite{DubickasJankauskas2015} and Virbalas \cite{Virbalas2025a}, implies that $d=20$.

\end{abstract}
\maketitle

\section{Introduction}\label{intro}

In 2004 Dubickas and Smyth \cite{DubickasSmyth2006} asked to prove or disprove the following: If $\alpha+\alpha'+\alpha''=0$ for three 
distinct algebraic conjugates of an algebraic number $\alpha$ of degree $d$, then 3 divides $d$. 
Stong (see \cite{DubickasSmyth2006}) provided a counterexample when $d=20$. More precisely, he showed that 
the irreducible polynomial
\begin{equation*}
x^{20} + 4\cdot5^{9}\cdot x^{10} + 16\cdot5^{15}
\end{equation*}
has three distinct roots, which sum to zero. 
% During the open problem session of the \emph{Number Theory and Polynomials} conference held in Bristol, UK, in 2006 (see, e.g., \cite{McKee_Smyth_2008}), Smyth asked 
Several authors (see, e.g., \cite{DubickasJankauskas2015,Virbalas2025a}) were interested in the following natural question: what is the smallest positive integer $d$, not a multiple of 3, for which there exists an algebraic number of degree $d$ such that some three of its conjugates sum to zero? The above-mentioned counterexample of Stong implies that $d\leq 20$. Dubickas and Jankauskas (see Theorem 1.2 in \cite{DubickasJankauskas2015}) showed that such a minimal value of $d$ lies in the range $10\leq d \leq 20$. Note that, by Lemma \ref{intro4} (see Section~\ref{intro2}), $d$ cannot be a prime number. 
Thus, $d=10, 14, 16$ or $20$. 
Recently, Virbalas (see Theorem 1.1 in \cite{Virbalas2025a}) showed that $d\neq 2p$, where $p\geq 5$ is a prime number. 
Thus, either $d=16$ or $d=20$. The main result of this paper states that $d=20$.

\begin{theorem}\label{t1}
Let $d$ be the smallest positive integer, not a multiple of $3$, for which there exists an algebraic number of degree $d$ whose some three algebraic conjugates sum to zero. Then $d=20$.
%Let $\alpha$ be an algebraic number over $\mathbb{Q}$ of degree $d$. Assume that $d$ is the smallest positive integer, not multiple of $3$, such that $\alpha_{i}+\alpha_{j}+\alpha_{k}=0$, for some three conjugates of $\alpha$. Then $d=20$.
\end{theorem}

The relation $\alpha+\alpha'+\alpha''=0$ between three algebraic conjugates is a particular case of a more general relation
\begin{equation}\label{eqin1}
a_1\alpha_1+a_2\alpha_2+\dotsb+a_n\alpha_n=0,
\end{equation}
where $\alpha_1,\alpha_2,\dotsc,\alpha_n$ are the algebraic conjugates of an algebraic number $\alpha$ of degree $n$ and 
$a_1,a_2,\dotsc,a_n$ are rational integers, not all zero. The relation \eqref{eqin1} is called \textit{trivial} if $a_1=a_2=\dotsc=a_n$. One of the first general results was obtained by Kurbatov \cite{Kurbatov1977}, who proved that there are 
no non-trivial relations \eqref{eqin1} if the degree $n$ is a prime number (see also \cite{Baron1995} and \cite{Dixon1997}). 
Girstmair in \cite{Girstmair1999} proposed a theoretical framework to study linear relations \eqref{eqin1}, based on representation theory of finite groups, applied to the Galois group of the Galois closure of $\mathbb{Q}(\alpha)$. 

A lot of attention was devoted to the investigation of the relation \eqref{eqin1} for small values of $n$ and $a_1,a_2,\dotsc,a_n$. See, e.g., \cite{Girstmair1982, Girstmair2006, Girstmair2007, Girstmair2008, Lalande2007, Lalande2010, DubickasHareJankauskas2017,DubickasJankauskas2015} for the results related to the linear relation $\alpha_1+\alpha_2=\alpha_3$,  \cite{Baronenas2026} for the results related to the linear relations $\alpha_1=\alpha_2+\alpha_3+\alpha_4$ and $\alpha_1+\alpha_2=\alpha_3+\alpha_4$, \cite{DuVi25,Kitaoka2017} for the classification of all possible relations in case when $n=4$ and \cite{Serrano2025} for the multiplicative analog.

In the proof of Theorem~\ref{t1} (see Section~\ref{result}), we assume that there exists an algebraic number $\alpha$ of degree 16 whose three algebraic conjugates sum to zero. Then we construct the graph $\mathcal{G}$ whose set of vertices 
is the set $\{\alpha_1,\alpha_2,\dotsc,\alpha_{16}\}$ of algebraic conjugates of $\alpha$. 
Two distinct vertices $\alpha'$ and $\alpha''$ are adjacent if and only if there exists a conjugate $\alpha'''$ such that $\alpha'+\alpha''+\alpha'''=0$. We prove that the graph $\mathcal{G}$ is vertex-transitive and every vertex has degree (valency) 6. Then we use the classification of such graphs. There are exactly 40 vertex-transitive graphs on 16 vertices of degree 6 (see, e.g., McKay and Royle \cite{mckay1990transitive}; the explicit list is maintained in \cite{HoltRoyle2020}\footnote{https://doi.org/10.5281/zenodo.4010122}). 
Then, using auxiliary results (see Section~\ref{intro2}, and Lemma~\ref{lemrc} and Lemma~\ref{lemzstp} in Section~\ref{result}) and computations with SageMath \cite{sagemath} (the code is provided in the appendix),  we show that neither of these 40 graphs is isomorphic to $\mathcal{G}$.  

\section{Auxiliary results}\label{intro2}

As far as we know, the first general result, related to linear relations among algebraic conjugates of a given algebraic number, was obtained by Kurbatov \cite{Kurbatov1977}:

\begin{lemma}\label{intro4}
The equality
\[
k_{1}\alpha_{1}+k_{2}\alpha_{2}+\cdots+k_{d}\alpha_{d} = 0
\]
with conjugates $\alpha_{1}, \alpha_{2},\dots, \alpha_{d}$ of an algebraic number $\alpha$ of prime degree $d$ over $\mathbb{Q}$ and $k_{1}, k_{2},\dots, k_{d}\in\mathbb{Z}$ can only hold if  $k_{1} = k_{2} = \cdots = k_{d}$.
\end{lemma}

The following result of Smyth \cite{Smyth1982} will be used several times in the proof of Theorem~\ref{t1} to eliminate impossible relations among algebraic conjugates.

\begin{lemma}\label{intro7}
If $\alpha_{1}, \alpha_{2}, \alpha_{3}$ are three conjugates of an algebraic number satisfying $\alpha_{1}\neq\alpha_{2}$ then $2\alpha_{1}\neq \alpha_{2} + \alpha_{3}$.
\end{lemma}

The following result, obtained by Dubickas \cite{Dubickas2002} (see also \cite[Theorem~3']{Dixon1997}), is a generalization of Lemma $\ref{intro7}$. 

\begin{lemma}\label{lemma4}
If $\beta_{1}, \beta_{2},\dots , \beta_{n}$, where $n\geq 3$, are distinct algebraic numbers conjugate over a field of characteristic zero $K$ and $k_{1}, k_{2},\dots ,k_{n}$ are non-zero rational numbers satisfying $|k_{1}| \geq |k_{2}|+\dots+|k_{n}|$ then
\begin{equation*}
k_{1}\beta_{1} + k_{2}\beta_{2} +\dots+ k_{n}\beta_{n}\notin K.
\end{equation*}
\end{lemma}

Also, we will use the following result, proved by Dubickas and Jankauskas in \cite{DubickasJankauskas2015}. 

\begin{lemma}\label{intro3}
The equality
\begin{equation*}
k_{1}\alpha_{1}+k_{2}\alpha_{2}+\cdots+k_{d}\alpha_{d} = 0
\end{equation*}
with conjugates $\alpha_{1}, \alpha_{2},\dots, \alpha_{d}$ of an algebraic number $\alpha$ of degree $d$ over $\mathbb{Q}$ and $k_{1}, k_{2},\dots, k_{d}\in\mathbb{Z}$ satisfying $\sum_{i=1}^{d} k_{i}\neq 0$ can only hold if $tr(\alpha) := \alpha_{1} + \alpha_{2} + \cdots + \alpha_{d} = 0$.
\end{lemma}

\section{Proof of Theorem $\ref{t1}$}\label{result}

% \begin{tikzpicture}[
%     scale=1.2, % Adjust this to scale the entire graph up or down
%     vertex/.style={circle, draw=black, thick, fill=red!15, inner sep=1.5pt, minimum size=18pt, font=\sffamily}
% ]
%     % Define the coordinates for each vertex
%     % Placing node 7 at the origin (0,0) as a reference point
%     \coordinate (V7) at (0, 0);
%     \coordinate (V3) at (-2.5, 0);
%     \coordinate (V6) at (0, 2.5);
%     \coordinate (V1) at (1.5, 1.2);
%     \coordinate (V2) at (2.5, 4);
%     \coordinate (V4) at (3.5, -0.6);
%     \coordinate (V5) at (1.7, -2.2);

%     % Draw the edges first so they sit strictly behind the nodes
%     % Edges connected to the central node 1
%     \draw[thick] (V1) -- (V2);
%     \draw[thick] (V1) -- (V3);
%     \draw[thick] (V1) -- (V4);
%     \draw[thick] (V1) -- (V5);
%     \draw[thick] (V1) -- (V6);
%     \draw[thick] (V1) -- (V7);
    
%     % Remaining edges
%     \draw[thick] (V2) -- (V6);
%     \draw[thick] (V3) -- (V7);
%     \draw[thick] (V6) -- (V7);
%     \draw[thick] (V7) -- (V4);
%     \draw[thick] (V7) -- (V5);
%     \draw[thick] (V4) -- (V5);

%     % Draw the nodes on top of the edges
%     \node[vertex] at (V1) {1};
%     \node[vertex] at (V2) {2};
%     \node[vertex] at (V3) {3};
%     \node[vertex] at (V4) {4};
%     \node[vertex] at (V5) {5};
%     \node[vertex] at (V6) {6};
%     \node[vertex] at (V7) {7};

% \end{tikzpicture}

\begin{proof}
Let $d$ be a positive integer, not divisible by $3$. Let $\alpha$ be an algebraic number over $\mathbb{Q}$ of degree $d$ whose some three algebraic conjugates sum to zero. Suppose that $d$ is the smallest possible such positive integer. 

Stong (see \cite{DubickasSmyth2006}) showed that  
the irreducible polynomial
\begin{equation*}
x^{20} + 4\cdot5^{9}\cdot x^{10} + 16\cdot5^{15}
\end{equation*}
has three distinct roots, which sum to zero. Hence, $d\leq 20$. 
On the other hand, Dubickas and Jankauskas (see Theorem 1.2 in \cite{DubickasJankauskas2015}) proved that $d\geq 10$. 
Moreover, by Lemma \ref{intro4}, $d$ cannot be a prime number. 
Thus $d\in\{10,14,16,20\}$. Recently, Virbalas (see Theorem 1.1 in \cite{Virbalas2025a}) proved that $d\neq 2p$, where $p\geq 5$ is a 
prime number. So $d=16$ or $d=20$. We will prove that $d\neq 16$.

Assume, to the contrary, that there exists an algebraic number $\alpha$ of degree $d=16$ whose some three algebraic conjugates $\alpha_{1},\alpha_{2},\alpha_{3}$  satisfy the relation 
\begin{equation}\label{eq1}
\alpha_{1}+\alpha_{2}+\alpha_{3}=0.
\end{equation}
Clearly, these three conjugates can't be all equal. 
Moreover, if some two conjugates are equal, say $\alpha_1=\alpha_2$, then 
$2\alpha_1+\alpha_3=0$, which is impossible (we can choose an 
automorphism $\pi$ of the Galois group of the normal closure of $\mathbb{Q}(\alpha)$ 
which maps $\alpha_1$ to $\alpha'$ having the maximal absolute value $m$; then 
$2\alpha_1+\alpha_3=0$ is mapped to $2\alpha'+\pi(\alpha_3)=0$ and $2m=|2\alpha'|=|-\pi(\alpha_3)|\leq m$, which is impossible). Hence, all three conjugates in \eqref{eq1} are distinct.

Let $\alpha_1,\alpha_2,\dotsc,\alpha_{16}$ be the algebraic conjugates of $\alpha$.  Lemma $\ref{intro3}$ implies that $tr(\alpha)=\alpha_{1}+\alpha_{2}+\dotsb+\alpha_{16}=0$. Let $G$ be the Galois group of the normal closure of $\mathbb{Q}(\alpha)$ over $\mathbb{Q}$. Note that this normal closure is also the splitting field of the minimal polynomial of $\alpha$ over $\mathbb{Q}$, and therefore $G$ is the Galois group of this polynomial. The group $G$ corresponds to some transitive subgroup of the full symmetric group $S_{16}$.

Two relations $\alpha_{i}+\alpha_{j}+\alpha_{k}=0$ and $\alpha_{i'}+\alpha_{j'}+\alpha_{k'}=0$, where all the summands are  algebraic 
conjugates of $\alpha$, are called distinct, if 
\begin{equation*}
\{\alpha_{i},\alpha_{j},\alpha_{k}\}\neq \{\alpha_{i'},\alpha_{j'},\alpha_{k'}\}.
\end{equation*}

\begin{lemma}\label{lemri}
For any two distinct relations $\alpha_{i}+\alpha_{j}+\alpha_{k}=0$ and $\alpha_{i'}+\alpha_{j'}+\alpha_{k'}=0$ we have that $|\{i, j, k\}\cap \{i', j', k'\}|\leq 1$.
\end{lemma}
\begin{proof}
Since the relations are distinct, we have that $|\{i, j, k\}\cap \{i', j', k'\}|\leq 2$. 
On the other hand, if say $\{j, k\}= \{j', k'\}$, then 
$\alpha_{i}=-\alpha_{j}-\alpha_{k}=-\alpha_{j'}-\alpha_{k'} = \alpha_{i'}$, which is impossible. Hence, the claim follows.
\end{proof}

\begin{lemma}\label{lemr6}
For any fixed $i_0\in\{1,2,\dotsc,16\}$ the number of distinct relations of the form $\alpha_{i_0}+\alpha_{j}+\alpha_{k}=0$ is less than 6.
\end{lemma}
\begin{proof}
Assume, to the contrary, that for a fixed $i_0$ there exist six distinct relations of the form $\alpha_{i_0}+\alpha_{j}+\alpha_{k}=0$. 
Lemma~\ref{lemri} implies that for any two distinct relations $\alpha_{i_0}+\alpha_{j_1}+\alpha_{k_1}=0$ and $\alpha_{i_0}+\alpha_{j_2}+\alpha_{k_2}=0$ 
the sets $\{j_1,k_1\}$ and $\{j_2,k_2\}$ are disjoint. Without loss of generality, we consider the following six relations (after relabeling the conjugates of $\alpha$, if necessary):
\begin{equation*}
\begin{split}
\alpha_{1}&+\alpha_{2}+\alpha_{3}=0,\\
\alpha_{1}&+\alpha_{4}+\alpha_{5}=0,\\
\alpha_{1}&+\alpha_{6}+\alpha_{7}=0,\\
\alpha_{1}&+\alpha_{8}+\alpha_{9}=0,\\ 
\alpha_{1}&+\alpha_{10}+\alpha_{11}=0,\\
\alpha_{1}&+\alpha_{12}+\alpha_{13}=0.
\end{split}
\end{equation*}
By adding all of them and using $tr(\alpha)=0$, we obtain
\begin{equation*}
6\alpha_{1}+\alpha_{2}+\alpha_{3}+\dots+\alpha_{13}=5\alpha_{1}-\alpha_{14}-\alpha_{15}-\alpha_{16}=0,
\end{equation*}
which is impossible by Lemma \ref{lemma4}.
\end{proof}

We say that two relations $\alpha_{i}+\alpha_{j}+\alpha_{k}=0$ and $\alpha_{i'}+\alpha_{j'}+\alpha_{k'}=0$ are \textit{conjugate} if 
there exists an automorphism $\pi\in G$ such that
\begin{equation*}
\{\alpha_{i},\alpha_{j},\alpha_{k}\} = \{\pi(\alpha_{i'}),\pi(\alpha_{j'}),\pi(\alpha_{k'})\}.
\end{equation*}
One can easily see that this conjugacy relation is an equivalence relation on the set of all possible relations of the form $\alpha_{i}+\alpha_{j}+\alpha_{k}=0$. 

\begin{lemma}\label{lemrc}
The following statements are true.
\begin{itemize}
\item[(i)] Any two relations of the form $\alpha_{i}+\alpha_{j}+\alpha_{k}=0$ are conjugate.
\item[(ii)] There are exactly 16 distinct such relations. 
\item[(iii)] Each algebraic conjugate of $\alpha$ appears in exactly 3 distinct such relations.
\end{itemize}
\end{lemma}
\begin{proof}
Consider a relation $\alpha_{i}+\alpha_{j}+\alpha_{k}=0$ and its equivalence class $\mathcal{C}(\alpha_{i},\alpha_{j},\alpha_{k})$, 
consisting of all the distinct relations that are conjugate to $\alpha_{i}+\alpha_{j}+\alpha_{k}=0$. Let $N=|\mathcal{C}(\alpha_{i},\alpha_{j},\alpha_{k})|$. 
Since the Galois group $G$ acts transitively on the set $\{\alpha_1,\alpha_2,\dotsc,\alpha_{16}\}$, each conjugate $\alpha_{i}$ appears an 
equal number of times, say $l$, in the relations in $\mathcal{C}(\alpha_{i},\alpha_{j},\alpha_{k})$. 
We have $N$ distinct relations in $\mathcal{C}(\alpha_{i},\alpha_{j},\alpha_{k})$ and each such relation involves three distinct conjugates of $\alpha$. 
Thus, there are $3N$ appearances of conjugates of $\alpha$ in $\mathcal{C}(\alpha_{i},\alpha_{j},\alpha_{k})$. 
On the other hand, this number can be counted in a different way -- each conjugate of $\alpha$ appears exactly $l$ times. 
Hence, $3N = 16l$. So that $l$ is divisible by 3. Since, by Lemma~\ref{lemr6}, $l<6$, we obtain that $l=3$ and accordingly $N=16$.

Assume that there are two distinct equivalence classes. Each contains three distinct relations involving $\alpha_1$. So that in total we have six distinct relations involving $\alpha_1$. This contradicts Lemma~\ref{lemr6}. Hence, the lemma follows.
\end{proof}

By Lemma~\ref{lemrc}, we have exactly 16 distinct relations of the form $\alpha_{i}+\alpha_{j}+\alpha_{k}=0$ and each $\alpha_i$ appears in exactly 3 distinct such relations.
Without loss of generality, we assume the following system of distinct relations:
\begin{equation}\label{eq2}
\begin{alignedat}{6}
&\alpha_{1} &  &+ &\;\; &\alpha_{2} & &+ &\;\; &\alpha_{3} & &= 0,\\
&\alpha_{1} &  &+ & &\alpha_{4} & &+ & &\alpha_{5} & &= 0,\\
&\alpha_{1} &  &+ & &\alpha_{6} & &+ & &\alpha_{7} & &= 0,\\
&\alpha_{i_{4\,1}} & &+ & &\alpha_{i_{4\,2}} & &+ & &\alpha_{i_{4\,3}} & &= 0,\\
&\alpha_{i_{5\,1}} & &+ & &\alpha_{i_{5\,2}} & &+ & &\alpha_{i_{5\,3}} & &= 0,\\
 & &          &\vdots &&&&\vdots &&&&\vdots\\
&\alpha_{i_{16\,1}} & &+ & &\alpha_{i_{16\,2}} & &+ & &\alpha_{i_{16\,3}} & &= 0.
\end{alignedat}
\end{equation}
This system of relations is equivalent to the following one:
\begin{equation}\label{eq3}
\begin{alignedat}{13}
-\alpha_{1}& &\;\; &= &\;\;&\alpha_{2} & &+&\;\; & \alpha_{3} & &= &\;\; &\alpha_{4} & &+ &\;\; &\alpha_{5}& &= &\;\; &\alpha_{6} & &+ &\;\; &\alpha_{7},\\
-\alpha_{2}& & &= &\;\;&\alpha_{1} & &+ & &\alpha_{3} & &= & &\alpha_{j_{2\,3}} & &+ & &\alpha_{j_{2\,4}}& &= & &\alpha_{j_{2\,5}} & &+ & &\alpha_{j_{2\,6}},\\
-\alpha_{3}& & &= &\;\;&\alpha_{1} & &+ & &\alpha_{2} & &= & &\alpha_{j_{3\,3}} & &+ & &\alpha_{j_{3\,4}} & &= & &\alpha_{j_{3\,5}} & &+ & &\alpha_{j_{3\,6}},\\
-\alpha_{4}& & &= &\;\;&\alpha_{1} & &+ & &\alpha_{5} & &= & &\alpha_{j_{4\,3}} & &+ & &\alpha_{j_{4\,4}} & & = & &\alpha_{j_{4\,5}} & &+ & &\alpha_{j_{4\,6}},\\
-\alpha_{5}& & &= &\;\;&\alpha_{1} & &+ & &\alpha_{4} & &= & &\alpha_{j_{5\,3}} & &+ & &\alpha_{j_{5\,4}} & &= & &\alpha_{j_{5\,5}} & &+ & &\alpha_{j_{5\,6}},\\
-\alpha_{6}& & &= &\;\;&\alpha_{1} & &+ & &\alpha_{7} & &= & &\alpha_{j_{6\,3}} & &+ & &\alpha_{j_{6\,4}} & &= & &\alpha_{j_{6\,5}} & &+ & &\alpha_{j_{6\,6}},\\
-\alpha_{7}& & &= &\;\;&\alpha_{1} & &+ & &\alpha_{6} & &= & &\alpha_{j_{7\,3}} & &+ & &\alpha_{j_{7\,4}} & &= & &\alpha_{j_{7\,5}} & &+ & &\alpha_{j_{7\,6}},\\
-\alpha_{8}& & &= &\;\;&\alpha_{j_{8\,1}} & &+ & &\alpha_{j_{8\,2}} & &= & &\alpha_{j_{8\,3}} & &+ & &\alpha_{j_{8\,4}} & &=& & \alpha_{j_{8\,5}} & &+ & &\alpha_{j_{8\,6}},\\
           & & &\vdots &\;\;&             & &   & &                 & &\vdots& &              & &     & &               & &\vdots& &           & & & & \\
-\alpha_{16}& & &= &\;\;&\alpha_{j_{16\,1}} & &+ & &\alpha_{j_{16\,2}} & &= & &\alpha_{j_{16\,3}} & &+ & &\alpha_{j_{16\,4}} & &= & &\alpha_{j_{16\,5}} & &+ & &\alpha_{j_{16\,6}}.
\end{alignedat}
\end{equation}

%Note that each $\alpha_{i}$ appears exactly six times on the right hand side in \eqref{eq3}. 
%We will show that \eqref{eq3} system (and, consequently, \eqref{eq2}) cannot have a solution. Assume for a contradiction that solution to \eqref{eq3} exists. 
Note that $-\alpha_{1}, -\alpha_{2},\dots, -\alpha_{16}$ are the algebraic conjugates of $-\alpha_{1}$. 
%Indeed, if $p(x)$ is the minimal polynomial of $\alpha$ over $\mathbb{Q}$ and $\deg(\alpha)=16$, then $p(-x)$ is the minimal polynomial of $-\alpha$ over $\mathbb{Q}$. 
This implies that $\deg(\alpha_{1}+\alpha_{2})=16$ and each sum $\alpha_{i}+\alpha_{j}$ appearing in \eqref{eq3} is an algebraic conjugate of $\alpha_{1}+\alpha_{2}$. 
%Moreover, it is a full set of algebraic conjugates of $\alpha_{1}+\alpha_{2}$. 
Note that each algebraic conjugate of $\alpha_{1}+\alpha_{2}$ has exactly three distinct representations as a sum $\alpha_{i}+\alpha_{j}$ (see $(iii)$ in Lemma~\ref{lemrc}).

Now, we will introduce the graph theory approach to this problem. 
Consider the graph $\mathcal{G}=(V, E)$ with the set of vertices $V:=\{\alpha_{1}, \alpha_{2}, \dots, \alpha_{16}\}$. 
Two vertices $\alpha_{i}$ and $\alpha_{j}$ are connected by an edge, denoted $(\alpha_{i}, \alpha_{j})$ (or $(\alpha_{j}, \alpha_{i})$), 
if the sum $\alpha_{i}+\alpha_{j}$ is an algebraic conjugate of $\alpha_{1}+\alpha_{2}$, i.e., 
the sum $\alpha_{i}+\alpha_{j}$ appears in \eqref{eq3}. 
Note that each vertex of $\mathcal{G}$ must have the same number of edges. 
In other words, graph $\mathcal{G}$ is regular. 
This follows directly from the fact that each $\alpha_{i}$ appears exactly six times  
in \eqref{eq3} (or, equivalently, each $\alpha_{i}$ appears in three distinct relations $\alpha_{i}+\alpha_{j}+\alpha_{k}=0$ and each such relation implies two edges $(\alpha_i,\alpha_j)$ and $(\alpha_i,\alpha_k)$, connecting the vertex $\alpha_i$), i.e., $\deg(v)=6$ for each $v\in V$. In such a setting, $\mathcal{G}$ must have 
$(16\cdot6)/2=48$ edges. Furthermore, since the Galois group $G$ is transitive on the set of conjugates $\{\alpha_{1}, \alpha_{2}, \dots, \alpha_{16}\}$ and, by Lemma~\ref{lemrc}, any two relations of the form $\alpha_{i}+\alpha_{j}+\alpha_{k}=0$ are conjugate, the graph 
$\mathcal{G}$ must be vertex-transitive. %Thus, $|V|=16, |E|=48$, and $\mathcal{G}$ is regular graph.   
%A simple search with SageMath [???] shows that
It is known that there are exactly $40$ vertex-transitive graphs on $16$ vertices of degree 
$6$. 
This was established by McKay and Royle \cite{mckay1990transitive} the complete enumeration of transitive graphs on at most 26 vertices. 
The explicit list is maintained in \cite{HoltRoyle2020}\footnote{https://doi.org/10.5281/zenodo.4010122} (see also the House of Graphs \cite{Coolsaet2023}). All 40 of these graphs, named $\mathcal{G}_{1},\mathcal{G}_{2},\dotsc, \mathcal{G}_{40}$, are given in Table~\ref{table:t1}. 
Note that these graphs are given in graph6 format (this format was invented by Brendan McKay \cite{McKayPiperno2014, McKayFormats} and is recognized by SageMath \cite{sagemath}).

%\begin{center}
%\scalebox{0.8}{      
%\vspace{0,2cm}

\begin{table}[h!]
\centering
\scalebox{0.85}{
\begin{tabular}{ |c|c| }
\hline
Graph & Code in $graph6$ format\\ 
\hline
\hline
$\mathcal{G}_{1}$ & OsaC???FzrMkYwXwJw@|?\\ 
$\mathcal{G}_{2}$ & OsaC???RxnNKYw$\setminus$WJs@\}?\\ 
$\mathcal{G}_{3}$ & OsaC???]ZrK\{XwFwB\{B\{?\\
$\mathcal{G}_{4}$ & OsaCB@\_EWrKrXeFwB\{B\{?\\
$\mathcal{G}_{5}$ & OsaKYCQGbRIjXKLWJEHqc\\
$\mathcal{G}_{6}$ & OsaKYCcQXRBKSbLDmTBi\_\\
$\mathcal{G}_{7}$ & OsaKYDDGgdLBUELQeibr?\\
$\mathcal{G}_{8}$ & OsaKYPDHOiDFEM[wMPRcg\\
$\mathcal{G}_{9}$ & OsaKg?dQGid$\setminus$[S[qLSQy\_\\
$\mathcal{G}_{10}$ & OsaKg?hSZEkkTIIdJWPyA\\
$\mathcal{G}_{11}$ & OsaKiCaC\textasciigrave RbMYKTYLabiC\\
$\mathcal{G}_{12}$ & OsaKiCdPPDaUBR]WNAboS\\
$\mathcal{G}_{13}$ & OsaSWSTOhDKiXQUEfBbr?\\
$\mathcal{G}_{14}$ & OsaSXCdOha\textasciigrave T[DRRNEAyC\\
$\mathcal{G}_{15}$ & OsaSXDCSXRbKTBIdmSBY@\\
$\mathcal{G}_{16}$ & OsaSYHBGpH\textasciigrave XDL]SNDBoK\\
$\mathcal{G}_{17}$ & Osedw?DOYFHBSZW$\setminus$Fg@w\textasciigrave \\
$\mathcal{G}_{18}$ & Osedw?HOYFGbSZW$\setminus$Fg@w\textasciigrave \\
$\mathcal{G}_{19}$ & Osedw@??pJhMP$\setminus$UWEjBpA\\
$\mathcal{G}_{20}$ & OsfDw?@G\textasciigrave bgmW$\setminus$RWFFAyA\\
\hline
\end{tabular}
%}
\quad
%\scalebox{0.8}{
\begin{tabular}{ |c|c| }
\hline
Graph & Code in $graph6$ format\\ 
\hline
\hline
$\mathcal{G}_{21}$ & OsfDw@@GXPCZP]TSFDayA\\
$\mathcal{G}_{22}$ & OsfLg?@WYHcZQYKXJcPuA\\
$\mathcal{G}_{23}$ & Osqsw@@AXCclW]USEXaxA\\
$\mathcal{G}_{24}$ & Osqsy@@GWRcdGtUEESqyA\\
$\mathcal{G}_{25}$ & OtaCXOTHOphhRKSsKTJcK\\
$\mathcal{G}_{26}$ & OtaLw?@OaRgn[S[Kdk?z@\\
$\mathcal{G}_{27}$ & OtaLw?@ObBiNQ[P[fg@x@\\
$\mathcal{G}_{28}$ & OtaLw?BOBBhNP[S[fa@y@\\
$\mathcal{G}_{29}$ & OtaLy@@AWJg[WFSFfg@x@\\
$\mathcal{G}_{30}$ & OtaLyD\textasciigrave SYQGhKBKB\textasciigrave eOXb\\
$\mathcal{G}_{31}$ & Ota$\setminus$W@@GYChJS]PZFc@u@\\
$\mathcal{G}_{32}$ & Otakw?@QYbKLPLOufc@u@\\
$\mathcal{G}_{33}$ & Otaky@@GWbhBPROlfc@u@\\
$\mathcal{G}_{34}$ & OtrTOGBW@\textasciigrave hIP]G|Bg\_tP\\
$\mathcal{G}_{35}$ & Ouj$\setminus$w?@?YBcMSUILIhPTI\\
$\mathcal{G}_{36}$ & O\{fL\_@HKQJCZCsBW\_pzp\_\\
$\mathcal{G}_{37}$ & O\{fL\_CCQP\textasciigrave GmCzB[C$\setminus$Joo\\
$\mathcal{G}_{38}$ & O\}akqPPWOV@iHIDHcROcj\\
$\mathcal{G}_{39}$ & O\textasciitilde{}aKYPDOxQBHHIGeacocj\\
$\mathcal{G}_{40}$ & O\textasciitilde{}z$\setminus$w?@?WB\_M?Z?V\_Fz\{?\\
\hline
\end{tabular}
}
\vspace{0,2cm}
\caption{All $40$ vertex-transitive graphs on $16$ vertices of degree $6$ (see \cite{HoltRoyle2020}).}
\label{table:t1}
\end{table}
%\vspace{0,2cm}
%}
%\end{center}

%We will show that $\mathcal{G}\neq\mathcal{G}_{i}$, for $i=1,2,\dots 40$. 
%We will show that $\mathcal{G}\not\cong\mathcal{G}_{i}$, for $i=1,2,\dots 40$, i.e., the graph $\mathcal{G}$ is not isomorphic to any of the graphs mentioned in Table~\ref{table:t1}.
The goal is to prove that for all $i=1,2,\dots 40$, graphs $\mathcal{G}$ and $\mathcal{G}_{i}$ are not isomorphic.

Note that a relation $\alpha_{i}+\alpha_{j}+\alpha_{k}=0$ in \eqref{eq2} implies that $\alpha_{i}+\alpha_{j}=-\alpha_{k}, \alpha_{i}+\alpha_{k}=-\alpha_{j}$, and $\alpha_{j}+\alpha_{k}=-\alpha_{i}$ are algebraic conjugates of $\alpha_{1}+\alpha_{2}=-\alpha_{3}$. In other words, edges $(\alpha_{i}, \alpha_{j}), (\alpha_{i}, \alpha_{k})$, and $(\alpha_{j}, \alpha_{k})$ form a triangle in our graph $\mathcal{G}$. 
On the other hand, if for some three distinct $\alpha_{a}, \alpha_{b}, \alpha_c$ we have that three 
edges $(\alpha_{a}, \alpha_{b}), (\alpha_{b}, \alpha_{c})$ and $(\alpha_{a}, \alpha_{c})$ belong to 
our graph $\mathcal{G}$, then not necessarily the relation $\alpha_{a}+ \alpha_{b}+\alpha_{c}=0$ holds. 
If it does, then we say that the vertices $\alpha_{a}, \alpha_{b}, \alpha_{c}$ form a \textit{zero-sum triangle} $\triangle(\alpha_{a}, \alpha_{b}, \alpha_{c})$.

We will recall several concepts from Graph Theory. Let $G=(V,E)$ be a graph. 
Two vertices that are connected directly by an edge  
are called \textit{adjacent vertices}.  
A vertex is said to be \textit{incident to an edge} (and the edge is incident to the vertex) if that vertex 
is one of the endpoints of the edge. 
The \textit{degree} (or \textit{valency}) of a vertex $v\in V$ is defined as the number of edges in $G$ that are incident to $v$, and denoted by $\deg_G(v)$. 
The set of all vertices that are adjacent to a given vertex $v\in V$, except for $v$ itself, is  
called \textit{the open neighborhood}  of $v$, and denoted by $N(v)$. 
Then the set $N[v]:=N(v)\cup\{v\}$ is called \textit{the closed neighborhood}  of $v$. 
For a given subset $V'$ of the set of vertices $V$ let $E'$ be a subset of the set of edges $E$ 
such that the edge $vv'\in E$ belongs to $E'$ if and only if both vertices $v$ and $v'$ belong to $V'$. 
Such a graph $\tilde{G}=(V',E')$ is called \textit{the induced graph} on the subset of vertices $V'$. 
For a given vertex $v$ denote by $G(v)$ the induced graph on the closed neighborhood $N[v]$. 
Note that for $v'\in G(v)$ the number $\deg_{G(v)}(v')$ is the degree of the vertex $v'$ with respect to the induced graph $G(v)$ while $\deg_{G}(v')$ is the degree of $v'$ with respect to the graph $G$.

\begin{lemma}\label{lemzstp}
The following statements are true for the graph $\mathcal{G}$.
\begin{itemize}
\item[(i)] Two distinct zero-sum triangles share at most one vertex.
\item[(ii)] Every edge $(\alpha,\alpha')$, $\alpha\neq\alpha'$, belongs to exactly one zero-sum triangle 
$\triangle(\alpha, \alpha', \alpha'')$ and the vertex $\alpha''$ belongs to the open neighborhood $N(\alpha)$ of $\alpha$.
\item[(iii)] If for some three distinct vertices $\alpha, \alpha', \alpha''$ we have that 
$\alpha', \alpha''\in N(\alpha)$, $\deg_{\mathcal{G}(\alpha)}(\alpha')=2$ and $(\alpha',\alpha'')$ is an edge in $\mathcal{G}$, then $\triangle(\alpha, \alpha', \alpha'')$ is a zero-sum triangle.
\item[(iv)] If $\triangle(\alpha, \alpha_a, \alpha_b)$ and $\triangle(\alpha, \alpha_c, \alpha_d)$ are 
two distinct zero-sum triangles, then either $(\alpha_a,\alpha_d)$ or $(\alpha_b,\alpha_c)$ is not an edge in $\mathcal{G}$.
\end{itemize}
\end{lemma}
\begin{proof}
Part $(i)$ follows from the definition of the graph $\mathcal{G}$ and Lemma~\ref{lemri}.

$(ii)$. Let $(\alpha,\alpha')$, $\alpha\neq\alpha'$, be an edge in $\mathcal{G}$. 
Then, according to the definition of $\mathcal{G}$, the sum $\alpha+\alpha'$ is an algebraic conjugate of 
$\alpha_1+\alpha_2=-\alpha_3$. 
Hence, $\alpha+\alpha'=-\alpha''$, where $\alpha''$ is an algebraic conjugate of $\alpha$. 
So that $\alpha+\alpha'+\alpha''=0$, which implies that the triangle $\triangle(\alpha, \alpha', \alpha'')$ 
is a zero-sum triangle. Part $(i)$ ensures that $\triangle(\alpha, \alpha', \alpha'')$ is the unique triangle containing the edge $(\alpha,\alpha')$. Moreover, $(\alpha,\alpha'')$ is also an edge in $\mathcal{G}$, since 
$\alpha+\alpha''=-\alpha'$ is an algebraic conjugate of $\alpha_1+\alpha_2=-\alpha_3$. 
So $\alpha''$ is adjacent to $\alpha$, $\alpha''\neq\alpha$,  and therefore $\alpha''\in N(\alpha)$. 

$(iii)$. Since $\alpha'\in N(\alpha)$, we have that $(\alpha,\alpha')$ is an edge in $\mathcal{G}$. 
Applying part $(ii)$ we obtain that there exists a vertex $\tilde{\alpha}\in N(\alpha)$ such that 
the triangle $\triangle(\alpha, \alpha', \tilde{\alpha})$ is a zero-sum triangle. 
This together with $\deg_{\mathcal{G}(\alpha)}(\alpha')=2$ imply that in the induced graph $\mathcal{G}(\alpha)$ 
the vertex $\alpha'$ is adjacent only to $\alpha$ and $\tilde{\alpha}$. 
On the other hand, the edge $(\alpha',\alpha'')$ is  in $\mathcal{G}$ and $\alpha', \alpha''\in N(\alpha)$. 
Thus, the edge $(\alpha',\alpha'')$ belongs to the induced graph $\mathcal{G}(\alpha)$, and therefore 
$\alpha'$ is adjacent to $\alpha''$ in $\mathcal{G}(\alpha)$. Hence, $\tilde{\alpha}=\alpha''$ and the claim follows.  

$(iv)$. Assume, to the contrary, that $\triangle(\alpha, \alpha_a, \alpha_b)$ and $\triangle(\alpha, \alpha_c, \alpha_d)$ are 
two distinct zero-sum triangles and both $(\alpha_a,\alpha_d)$ and $(\alpha_b,\alpha_c)$ are edges in $\mathcal{G}$. 
Part $(ii)$ implies the existence of zero-sum triangles $\triangle(\alpha_a, \alpha_d, \alpha')$ and $\triangle(\alpha_b, \alpha_c, \alpha'')$. Note that $\alpha'\neq \alpha$, since, in view of part $(i)$, zero-sum triangles 
$\triangle(\alpha, \alpha_a, \alpha_b)$ and $\triangle(\alpha_a, \alpha_d, \alpha')$ share exactly one vertex $\alpha_a$. 
Similarly, $\alpha''\neq \alpha$. Now, the four obtained zero-sum triangles imply the relations
\begin{equation}\label{eql4}
\begin{split}
\alpha &+ \alpha_a + \alpha_b = 0,\\
\alpha &+ \alpha_c + \alpha_d = 0,\\
\alpha_a &+ \alpha_d + \alpha' = 0,\\
\alpha_b &+ \alpha_c + \alpha'' = 0.
\end{split}
\end{equation}
By adding the first two relations in \eqref{eql4} and using the expressions for $\alpha'$ and $\alpha''$ 
from the last two relations in \eqref{eql4}, we obtain
\begin{equation*}
2\alpha+ (\alpha_a+\alpha_d)+(\alpha_b+\alpha_c)=2\alpha-\alpha'-\alpha''=0.
\end{equation*}
This contradicts Lemma~\ref{intro7}, since $\alpha'\neq \alpha$.
%If $\alpha'=\alpha''$, then the last equality implies $\alpha=\alpha'$, 
%which is impossible. Hence, the set 
%$\{\alpha,\alpha',\alpha''\}$ contains at least two distinct numbers. 
%But then the relation $2\alpha-\alpha'-\alpha''=0$ contradicts Lemma~\ref{intro7}.
\end{proof}

%Note that a relation $\alpha_{i}+\alpha_{j}+\alpha_{k}=0$ in \eqref{eq2} implies that $\alpha_{i}+\alpha_{j}=-\alpha_{k}, \alpha_{i}+\alpha_{k}=-\alpha_{j}$, and $\alpha_{j}+\alpha_{k}=-\alpha_{i}$ are algebraic conjugates of $\alpha_{1}+\alpha_{2}=-\alpha_{3}$. In other words, edges $(\alpha_{i}, \alpha_{j}), (\alpha_{i}, \alpha_{k})$, and $(\alpha_{j}, \alpha_{k})$ form a triangle in our graph $\mathcal{G}$, denote it by $\triangle(\alpha_{i}, \alpha_{j}, \alpha_{k})$. We say that edge $(\alpha_{i}, \alpha_{j})$ comes from a triangle $\triangle(\alpha_{i}, \alpha_{j}, \alpha_{k})$ if and only if $\alpha_{i}+\alpha_{j}+\alpha_{k}=0$ is in \eqref{eq2}. Denote this by $(\alpha_{i}, \alpha_{j})\rightarrow\triangle(\alpha_{i}, \alpha_{j}, \alpha_{k})$. Notice that if $(\alpha_{i}, \alpha_{j})\rightarrow\triangle(\alpha_{i}, \alpha_{j}, \alpha_{k})$ and $(\alpha_{i}, \alpha_{j})\rightarrow\triangle(\alpha_{i}, \alpha_{j}, \alpha_{l})$, where $k\neq l$, then $\alpha_{k}=\alpha_{l}$, which is impossible. 
%Therefore, any edge $(\alpha_{i}, \alpha_{j})$ in graph $\mathcal{G}$ comes from the only one triangle.

%First, consider the following graph (FyTAG in graph6 format):
First, consider the graph shown in Figure~\ref{fig10}.

%\begin{center}
%\includegraphics[scale=0.12]{./Images/graph00}
%\end{center}

\begin{figure}[h]
\centering
% \begin{tikzpicture}[scale=0.6]
%     \tikzset{VertexStyle/.style={
%         shape=circle,
%         draw=black,
%         %thick,
%         %fill=red!15,
%         inner sep=1.2pt,
%         minimum size=8pt,
%         %font=\sffamily
%     }}
%     \tikzset{EdgeStyle/.style={thin}}

%     % Define the central vertex
%     \Vertex[L={}, x=0.0, y=0.0]{1}

%     % Define vertices for the left triangle
%     \Vertex[L={}, x=-3.0, y=0.0]{2}
%     \Vertex[L={}, x=-1.5, y=2.6]{3}

%     % Define vertices for the right triangle
%     \Vertex[L={}, x=1.5, y=2.6]{4}
%     \Vertex[L={}, x=3.0, y=0.0]{5}

%     % Define vertices for the bottom triangle
%     \Vertex[L={}, x=1.5, y=-2.6]{6}
%     \Vertex[L={}, x=-1.5, y=-2.6]{7}

%     % Draw edges for the left triangle
%     \Edge(1)(2)
%     \Edge(1)(3)
%     \Edge(2)(3)

%     % Draw edges for the right triangle
%     \Edge(1)(4)
%     \Edge(1)(5)
%     \Edge(4)(5)

%     % Draw edges for the bottom triangle
%     \Edge(1)(6)
%     \Edge(1)(7)
%     \Edge(6)(7)

% \end{tikzpicture}

\begin{tikzpicture}[scale=0.6]
    % Updated VertexStyle to perfectly match the solid black dots in the image
    \tikzset{VertexStyle/.style={
        shape=circle,
        fill=black,
        inner sep=0pt,
        minimum size=5pt,
        text height=0pt,    % Collapses the text box vertically
        text depth=0pt
    }}
    \tikzset{EdgeStyle/.style={thin}}

    \SetVertexNoLabel

    % Define the central vertex
    \Vertex[L={}, x=0.0, y=0.0]{1}

    % Define vertices for the top-left triangle
    \Vertex[L={}, x=-3.0, y=0.0]{2}
    \Vertex[L={}, x=-1.5, y=2.6]{3}

    % Define vertices for the top-right triangle
    \Vertex[L={}, x=1.5, y=2.6]{4}
    \Vertex[L={}, x=3.0, y=0.0]{5}

    % Define vertices for the bottom triangle
    \Vertex[L={}, x=1.5, y=-2.6]{6}
    \Vertex[L={}, x=-1.5, y=-2.6]{7}

    % Draw edges for the top-left triangle
    \Edge(1)(2)
    \Edge(1)(3)
    \Edge(2)(3)

    % Draw edges for the top-right triangle
    \Edge(1)(4)
    \Edge(1)(5)
    \Edge(4)(5)

    % Draw edges for the bottom triangle
    \Edge(1)(6)
    \Edge(1)(7)
    \Edge(6)(7)

\end{tikzpicture}
\caption{Graph FyTAG in graph6 format.}
\label{fig10}
\end{figure}
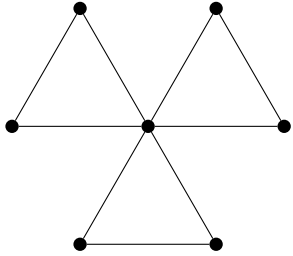

Notice that the first three equations in \eqref{eq2} imply that this graph must be a subgraph of $\mathcal{G}$. 
We check with SageMath \cite{sagemath} that neither of the following $29$ graphs
\begin{equation*}
\mathcal{G}_{1}, \mathcal{G}_{2},\dots, \mathcal{G}_{19};\quad\mathcal{G}_{25}, \mathcal{G}_{26},\dots, \mathcal{G}_{33};\quad\mathcal{G}_{39}
\end{equation*}
has a subgraph isomorphic to the graph FyTAG. Hence, $\mathcal{G}$ is isomorphic to one of the graphs in the following set
\begin{equation}\label{eq11a}
\{ \mathcal{G}_{20}, \mathcal{G}_{21}, \mathcal{G}_{22}, \mathcal{G}_{23}, \mathcal{G}_{24}, \mathcal{G}_{34},\mathcal{G}_{35}, \mathcal{G}_{36},\mathcal{G}_{37}, \mathcal{G}_{38}, \mathcal{G}_{40}\}.
\end{equation}

We claim that $\mathcal{G}$ is not isomorphic to any graph in the set $\{\mathcal{G}_{34}, \mathcal{G}_{36}, \mathcal{G}_{37}, \mathcal{G}_{38}\}$. Indeed, consider the graph shown in Figure~\ref{fig11}. %following graph (F\}TJG in graph6 format):

%\begin{equation*}
%\includegraphics[scale=0.3]{./Images/graph01}
%\end{equation*}

\begin{figure}[h]
\centering
\begin{tikzpicture}[scale=0.6]
    \tikzset{VertexStyle/.style={
        shape=circle,
        draw=black,
        %thick,
        %fill=red!15,
        inner sep=1.2pt,
        minimum size=8pt,
    }}
    \tikzset{EdgeStyle/.style={thin}}

    % Define vertices with coordinates forming a regular hexagon
    \Vertex[L={$\alpha_{k_1}$}, x=0.0, y=0.0]{1}
    \Vertex[L={$\alpha_{k_2}$}, x=-3.0, y=0.0]{2}
    \Vertex[L={$\alpha_{k_3}$}, x=-1.5, y=2.6]{3}
    \Vertex[L={$\alpha_{k_4}$}, x=1.5, y=2.6]{4}
    \Vertex[L={$\alpha_{k_5}$}, x=3.0, y=0.0]{5}
    \Vertex[L={$\alpha_{k_6}$}, x=1.5, y=-2.6]{6}
    \Vertex[L={$\alpha_{k_7}$}, x=-1.5, y=-2.6]{7}

    % Draw edges connected to the central node
    \Edge(1)(2)
    \Edge(1)(3)
    \Edge(1)(4)
    \Edge(1)(5)
    \Edge(1)(6)
    \Edge(1)(7)

    % Draw the outer hexagonal cycle edges
    \Edge(2)(3)
    \Edge(3)(4)
    \Edge(4)(5)
    \Edge(5)(6)
    \Edge(6)(7)
    \Edge(7)(2)

\end{tikzpicture}
\caption{Graph F\}TJG in graph6 format.}
\label{fig11}
\end{figure}
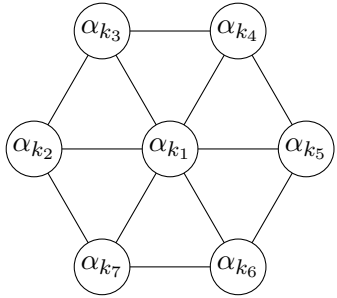

We will prove that this graph cannot be an induced subgraph of $\mathcal{G}$. 
Assume, to the contrary, that the graph F\}TJG is an induced subgraph of $\mathcal{G}$. 
By $(ii)$ of Lemma~\ref{lemzstp}, the edge $(\alpha_{k_{1}},\alpha_{k_{2}})$ belongs to 
a zero-sum triangle $\triangle(\alpha_{k_{1}}, \alpha_{k_{2}}, \alpha')$, where $\alpha'\in N(\alpha_{k_1})$. Since $\alpha_{k_{3}}$ and $\alpha_{k_{7}}$ are the only 
vertices in $N(\alpha_{k_{1}})$ that are adjacent to $\alpha_{k_{2}}$, we obtain that 
$\alpha'=\alpha_{k_{3}}$ or $\alpha_{k_{7}}$. 
Without loss of generality, assume $\alpha'=\alpha_{k_{3}}$. 
Thus, $\triangle(\alpha_{k_{1}}, \alpha_{k_{2}}, \alpha_{k_{3}})$ is a zero-sum triangle. 
Now part $(i)$ of Lemma~\ref{lemzstp} implies that neither $\triangle(\alpha_{k_{1}}, \alpha_{k_{3}}, \alpha_{k_{4}})$ nor $\triangle(\alpha_{k_{1}}, \alpha_{k_{2}}, \alpha_{k_{7}})$ is a zero-sum triangle. 
Applying $(ii)$ of Lemma~\ref{lemzstp}, we obtain that the edges $(\alpha_{k_{3}},\alpha_{k_{4}})$ 
and $(\alpha_{k_{2}},\alpha_{k_{7}})$ belong to zero-sum triangles 
$\triangle(\alpha_{k_{3}}, \alpha_{k_{4}}, \tilde{\alpha})$ and 
$\triangle(\alpha_{k_{2}}, \alpha_{k_{7}}, \alpha'')$, respectively, and $\tilde{\alpha},\alpha''\notin N[\alpha_{k_1}]=\{\alpha_{k_1},\alpha_{k_2},\dotsc,\alpha_{k_7}\}$. 

Analogously, applying $(ii)$ of Lemma~\ref{lemzstp}, we obtain that  
$\triangle(\alpha_{k_{1}}, \alpha_{k_{4}}, \alpha_{k_{5}})$ and 
$\triangle(\alpha_{k_{1}}, \alpha_{k_{6}}, \alpha_{k_{7}})$ 
are zero-sum triangles, 
$\triangle(\alpha_{k_{1}}, \alpha_{k_{5}}, \alpha_{k_{6}})$ is not a zero-sum triangle 
and there exists a vertex $\alpha'''$ such that 
$\triangle(\alpha_{k_{5}}, \alpha_{k_{6}}, \alpha''')$ is a zero-sum triangle 
and $\alpha'''\notin N[\alpha_{k_1}]$.

Now, all obtained zero-sum triangles imply the relations
\begin{equation}\label{eq4}
\begin{split}
\alpha_{k_{1}} &+ \alpha_{k_{2}} + \alpha_{k_{3}} = 0,\\
\alpha_{k_{1}} &+ \alpha_{k_{4}} + \alpha_{k_{5}} = 0,\\
\alpha_{k_{1}} &+ \alpha_{k_{6}} + \alpha_{k_{7}} = 0,\\
\alpha_{k_{3}} &+ \alpha_{k_{4}} + \tilde{\alpha} = 0,\\
\alpha_{k_{5}} &+ \alpha_{k_{6}} + \alpha''' = 0,\\
\alpha_{k_{2}} &+ \alpha_{k_{7}} + \alpha'' = 0.
\end{split}
\end{equation}
By adding the first three relations in \eqref{eq4} and using the expressions for $\tilde{\alpha},\alpha'',$ $\alpha'''$ from the last three relations in \eqref{eq4}, we obtain
\begin{equation*}
3\alpha_{k_{1}}+(\alpha_{k_{3}}+\alpha_{k_{4}})+(\alpha_{k_{5}}+\alpha_{k_{6}})+(\alpha_{k_{2}}+ \alpha_{k_{7}})=3\alpha_{k_{1}}-\tilde{\alpha}-\alpha'''-\alpha''=0.
\end{equation*}
If $\tilde{\alpha}=\alpha''=\alpha'''$, then the last equality implies $\alpha_{k_{1}}=\tilde{\alpha}$, 
which is impossible, since $\tilde{\alpha}\notin N[\alpha_{k_1}]$. Hence, the set 
$\{\alpha_{k_1},\tilde{\alpha},\alpha'',\alpha'''\}$ contains at least three distinct numbers. 
But then the relation $3\alpha_{k_{1}}-\tilde{\alpha}-\alpha'''-\alpha''=0$ contradicts Lemma~\ref{lemma4}.
This proves that the graph F\}TJG cannot be an induced subgraph of $\mathcal{G}$.

We check with SageMath \cite{sagemath} that each graph in $\{\mathcal{G}_{34}, \mathcal{G}_{36}, \mathcal{G}_{37}, \mathcal{G}_{38}\}$ contains an induced subgraph isomorphic to F\}TJG. 
%Therefore, $\mathcal{G}\notin\{\mathcal{G}_{34}, \mathcal{G}_{36}, \mathcal{G}_{37}, \mathcal{G}_{38}\}$.
%Therefore, $\mathcal{G}$ is not isomorphic to any graph in the set $\{\mathcal{G}_{34}, \mathcal{G}_{36}, \mathcal{G}_{37}, \mathcal{G}_{38}\}$.
Hence, $\mathcal{G}$ is not isomorphic to either of these graphs.
This together with \eqref{eq11a} imply that $\mathcal{G}$ must be isomorphic to some graph in the following set
\begin{equation*}
\{\mathcal{G}_{20}, \mathcal{G}_{21}, \mathcal{G}_{22}, \mathcal{G}_{23}, \mathcal{G}_{24}, \mathcal{G}_{35}, \mathcal{G}_{40}\}.
\end{equation*}
%\begin{equation*}
%\mathcal{G}\in\{ \mathcal{G}_{20}, \mathcal{G}_{21}, \mathcal{G}_{22}, \mathcal{G}_{23}, \mathcal{G}_{24}, \mathcal{G}_{35}, \mathcal{G}_{40}\}.
%\end{equation*}
%We will eliminate all these graphs one by one. 

%Next, we will prove that $\mathcal{G}\notin\{ \mathcal{G}_{20}, \mathcal{G}_{21}, \mathcal{G}_{22}, \mathcal{G}_{23}, \mathcal{G}_{24}\}$.
Next, we will prove that $\mathcal{G}$ is not isomorphic to either of the first five graphs in the set above, namely, $\{\mathcal{G}_{20}, \mathcal{G}_{21}, \mathcal{G}_{22}, \mathcal{G}_{23}, \mathcal{G}_{24}\}$.

The 16 vertices of each graph $\mathcal{G}_i$, $i\in\{20,21,22,23,24\}$, are labeled with numbers $1,2,\dotsc,16$.  
We will consider every such graph $\mathcal{G}_i$ separately. 
Assuming that $\mathcal{G}$ is isomorphic to $\mathcal{G}_i$, we will show that it leads to a contradiction. 
We will identify each vertex $v\in \{1,2,\dotsc,16\}$ of $\mathcal{G}_i$ with $\alpha_v$. 
So that instead of treating 
$\mathcal{G}$ and $\mathcal{G}_i$ as isomorphic graphs, we will treat them as coinciding graphs, 
i.e., $\mathcal{G}=\mathcal{G}_i$.

Assume that $\mathcal{G}=\mathcal{G}_{20}$. The induced subgraphs $\mathcal{G}_{20}(\alpha_1)$ and 
$\mathcal{G}_{20}(\alpha_6)$ are shown in Figure~\ref{fig3}. 
%\begin{figure}[h]
%\centering
%\includegraphics[scale=0.5]{./Images/G20i1}\;\;\;
%\includegraphics[scale=0.5]{./Images/G20i6}
%\caption{Induced subgraphs $\mathcal{G}_{20}(\alpha_1)$ and $\mathcal{G}_{20}(\alpha_6)$.}
%\label{fig3}
%\end{figure}

\begin{figure}[h]
\centering
\begin{tikzpicture}[scale=0.6]
    % Set the global style for the vertices
    \tikzset{VertexStyle/.style={
        shape=circle,
        draw=black,
        %thick,
        %fill=red!15,
        inner sep=1.2pt,
        minimum size=8pt,
    }}
    
    % Edge style configured to 'thin'
    \tikzset{EdgeStyle/.style={thin}}

    % Define vertices with explicit coordinates matching the visual layout
    \Vertex[x=0.0, y=0.0]{1}
    \Vertex[x=-1.8, y=1.6]{2}
    \Vertex[x=3.2, y=1.0]{3}
    \Vertex[x=1.5, y=-2.8]{4}
    \Vertex[x=-0.4, y=-2.5]{5}
    \Vertex[x=1.0, y=1.8]{6}
    \Vertex[x=1.6, y=-0.3]{7}

    % Draw edges connected to the central node 1
    \Edge(1)(2)
    \Edge(1)(3)
    \Edge(1)(4)
    \Edge(1)(5)
    \Edge(1)(6)
    \Edge(1)(7)

    % Draw remaining outer edges
    \Edge(2)(6)
    \Edge(3)(7)
    \Edge(4)(5)
    \Edge(4)(7)
    \Edge(5)(7)
    \Edge(6)(7)

\end{tikzpicture}\qquad\quad
\begin{tikzpicture}[scale=0.6]
    \tikzset{VertexStyle/.style={
        shape=circle,
        draw=black,
        %thick,
        %fill=red!15,
        inner sep=1.2pt,
        minimum size=8pt,
    }}
    \tikzset{EdgeStyle/.style={thin}}

    % Define vertices with explicit coordinates matching the visual layout
    \Vertex[x=0.0, y=0.0]{6}
    \Vertex[x=-0.4, y=2.0]{2}
    \Vertex[x=-2.8, y=0.0]{1}
    \Vertex[x=-1.2, y=-2.0]{7}
    \Vertex[x=-2.6, y=2.5]{12}
    \Vertex[x=2.0, y=3.2]{14}
    \Vertex[x=3.0, y=1.0]{16}

    % Draw edges connected to the central node 6
    \Edge(6)(1)
    \Edge(6)(2)
    \Edge(6)(7)
    \Edge(6)(12)
    \Edge(6)(14)
    \Edge(6)(16)

    % Draw edges connected to node 2
    \Edge(2)(1)
    \Edge(2)(12)
    \Edge(2)(14)
    \Edge(2)(16)

    % Draw remaining outer edges
    \Edge(1)(7)
    \Edge(14)(16)

\end{tikzpicture}
\caption{Induced subgraphs $\mathcal{G}_{20}(\alpha_1)$ and $\mathcal{G}_{20}(\alpha_6)$.}
\label{fig3}
\end{figure}
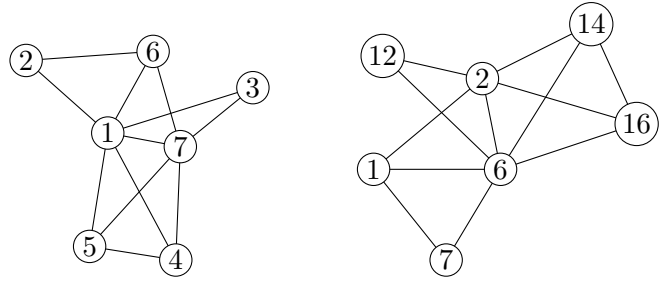

We have that $\alpha_2,\alpha_6\in N(\alpha_1)$, 
$\deg_{\mathcal{G}_{20}(\alpha_1)}(\alpha_2)=2$ and $(\alpha_2,\alpha_6)$ is an edge in $\mathcal{G}_{20}$. 
Hence, part $(iii)$ of Lemma~\ref{lemzstp} implies that $\triangle(\alpha_1,\alpha_2,$ $\alpha_6)$ is a zero-sum triangle. 
Similarly, $\alpha_1,\alpha_7\in N(\alpha_6)$, 
$\deg_{\mathcal{G}_{20}(\alpha_6)}(\alpha_7)=2$ and $(\alpha_1,\alpha_7)$ is an edge in $\mathcal{G}_{20}$.
Again, applying part $(iii)$ of Lemma~\ref{lemzstp}, we obtain that  $\triangle(\alpha_1,\alpha_6,\alpha_7)$ is a zero-sum triangle. So that zero-sum triangles $\triangle(\alpha_1,\alpha_2,\alpha_6)$ and $\triangle(\alpha_1,\alpha_6,\alpha_7)$ 
share two distinct vertices $\alpha_1$ and $\alpha_6$. This contradicts part $(i)$ of Lemma~\ref{lemzstp}. 
Thus, $\mathcal{G}\neq \mathcal{G}_{20}$. 

Analogously, for every $i\in\{21,22,23,24\}$, considering the induced subgraphs $\mathcal{G}_{i}(\alpha_1)$ and $\mathcal{G}_{i}(\alpha_6)$, which are given in Figures~\ref{fig4}-\ref{fig7}, and applying Lemma~\ref{lemzstp}, we obtain that $\mathcal{G}\neq \mathcal{G}_i$. 
%$\mathcal{G}\notin\{\mathcal{G}_{21}, \mathcal{G}_{22}, \mathcal{G}_{23}, \mathcal{G}_{24}\}$.
%\begin{figure}[h]
%\centering
%\includegraphics[scale=0.5]{./Images/G21i1}\;\;\;
%\includegraphics[scale=0.5]{./Images/G21i6}
%\caption{Induced subgraphs $\mathcal{G}_{21}(\alpha_1)$ and $\mathcal{G}_{21}(\alpha_6)$.}
%\label{fig4}
%\end{figure}

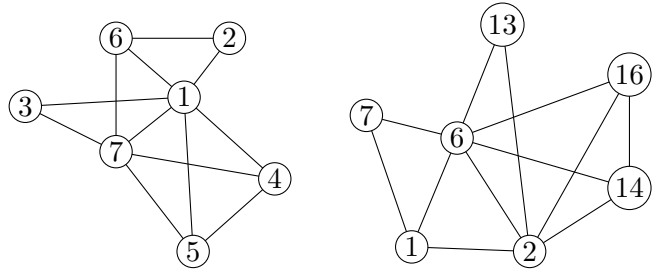
\begin{figure}[h]
\centering
\begin{tikzpicture}[scale=0.6]
    % Set the global style for the vertices to match the pale pink aesthetic
    \tikzset{VertexStyle/.style={
        shape=circle,
        draw=black,
        %thick,
        %fill=red!15,
        inner sep=1.2pt,
        minimum size=8pt,
    }}
    
    % Ensure edges are drawn thick to match the image
    \tikzset{EdgeStyle/.style={thin}}

    % Define vertices with explicit coordinates
    \Vertex[x=1.5, y=1.2]{1}
    \Vertex[x=2.5, y=2.5]{2}
    \Vertex[x=-2.0, y=1.0]{3}
    \Vertex[x=3.5, y=-0.6]{4}
    \Vertex[x=1.7, y=-2.2]{5}
    \Vertex[x=0.0, y=2.5]{6}
    \Vertex[x=0.0, y=0.0]{7}

    % Draw the edges
    % tkz-graph automatically handles drawing edges behind the nodes
    
    % Edges connected to central node 1
    \Edge(1)(2)
    \Edge(1)(3)
    \Edge(1)(4)
    \Edge(1)(5)
    \Edge(1)(6)
    \Edge(1)(7)

    % Remaining edges
    \Edge(2)(6)
    \Edge(3)(7)
    \Edge(6)(7)
    \Edge(7)(4)
    \Edge(7)(5)
    \Edge(4)(5)

\end{tikzpicture}\qquad
\begin{tikzpicture}[scale=0.6]
    % Set the global style for the vertices
    \tikzset{VertexStyle/.style={
        shape=circle,
        draw=black,
        %thick,
        %fill=red!15,
        inner sep=1.2pt,
        minimum size=8pt,
    }}
    
    % Edge style (change 'thick' to 'thin' if you want narrower lines)
    \tikzset{EdgeStyle/.style={thin}}

    % Define vertices with explicit coordinates to match the image layout
    \Vertex[x=0.0, y=0.0]{6}
    \Vertex[x=1.6, y=-2.5]{2}
    \Vertex[x=1.0, y=2.5]{13}
    \Vertex[x=3.8, y=1.4]{16}
    \Vertex[x=3.8, y=-1.1]{14}
    \Vertex[x=-1.0, y=-2.4]{1}
    \Vertex[x=-2.0, y=0.5]{7}

    % Draw edges connected to node 6
    \Edge(6)(1)
    \Edge(6)(2)
    \Edge(6)(7)
    \Edge(6)(13)
    \Edge(6)(14)
    \Edge(6)(16)

    % Draw remaining outer edges
    \Edge(7)(1)
    \Edge(1)(2)
    \Edge(13)(2)
    \Edge(2)(16)
    \Edge(2)(14)
    \Edge(16)(14)

\end{tikzpicture}

\caption{Induced subgraphs $\mathcal{G}_{21}(\alpha_1)$ and $\mathcal{G}_{21}(\alpha_6)$.}
\label{fig4}
\end{figure}

%\begin{figure}[h]
%\centering
%\includegraphics[scale=0.5]{./Images/G22i1}\;\;\;
%\includegraphics[scale=0.5]{./Images/G22i6}
%\caption{Induced subgraphs $\mathcal{G}_{22}(\alpha_1)$ and $\mathcal{G}_{22}(\alpha_6)$.}
%\label{fig5}
%\end{figure}

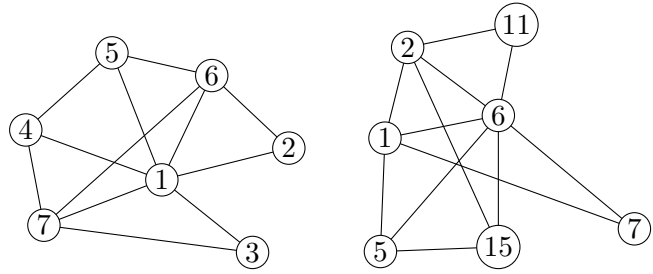
\begin{figure}[h] %G22
\centering
\begin{tikzpicture}[scale=0.6]
    \tikzset{VertexStyle/.style={
        shape=circle,
        draw=black,
        %thick,
        %fill=red!15,
        inner sep=1.2pt,
        minimum size=8pt,
    }}
    \tikzset{EdgeStyle/.style={thin}}

    % Define vertices with coordinates matching the visual layout
    \Vertex[x=0.0, y=0.0]{1}
    \Vertex[x=2.8, y=0.7]{2}
    \Vertex[x=2.0, y=-1.6]{3}
    \Vertex[x=-3.0, y=1.1]{4}
    \Vertex[x=-1.1, y=2.8]{5}
    \Vertex[x=1.1, y=2.3]{6}
    \Vertex[x=-2.6, y=-1.0]{7}

    % Draw edges connected to the central node 1
    \Edge(1)(2)
    \Edge(1)(3)
    \Edge(1)(4)
    \Edge(1)(5)
    \Edge(1)(6)
    \Edge(1)(7)

    % Draw remaining outer edges
    \Edge(2)(6)
    \Edge(3)(7)
    \Edge(4)(5)
    \Edge(4)(7)
    \Edge(5)(6)
    \Edge(6)(7)

\end{tikzpicture}\qquad
\begin{tikzpicture}[scale=0.6]
    \tikzset{VertexStyle/.style={
        shape=circle,
        draw=black,
        %thick,
        %fill=red!15,
        inner sep=1.2pt,
        minimum size=8pt,
    }}
    \tikzset{EdgeStyle/.style={thin}}

    % Define vertices with explicit coordinates matching the visual layout
    \Vertex[x=0.0, y=0.0]{6}
    \Vertex[x=-2.5, y=-0.5]{1}
    \Vertex[x=-2.0, y=1.5]{2}
    \Vertex[x=0.4, y=2.0]{11}
    \Vertex[x=3.0, y=-2.5]{7}
    \Vertex[x=0.0, y=-2.9]{15}
    \Vertex[x=-2.6, y=-3.0]{5}

    % Draw edges connected to the central node 6
    \Edge(6)(1)
    \Edge(6)(2)
    \Edge(6)(5)
    \Edge(6)(7)
    \Edge(6)(11)
    \Edge(6)(15)

    % Draw edges connected to node 1
    \Edge(1)(2)
    \Edge(1)(5)
    \Edge(1)(7)

    % Draw remaining outer edges
    \Edge(2)(11)
    \Edge(2)(15)
    \Edge(5)(15)

\end{tikzpicture}
\caption{Induced subgraphs $\mathcal{G}_{22}(\alpha_1)$ and $\mathcal{G}_{22}(\alpha_6)$.}
\label{fig5}
\end{figure}

%\begin{figure}[h]
%\centering
%\includegraphics[scale=0.5]{./Images/G23i1}\;\;\;
%\includegraphics[scale=0.5]{./Images/G23i6}
%\caption{Induced subgraphs $\mathcal{G}_{23}(\alpha_1)$ and $\mathcal{G}_{23}(\alpha_6)$.}
%\label{fig6}
%\end{figure}

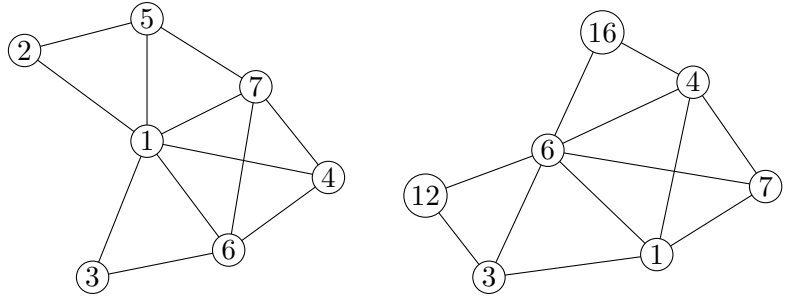
\begin{figure}[h] %G23

\begin{tikzpicture}[scale=0.6]
    \tikzset{VertexStyle/.style={
        shape=circle,
        draw=black,
        %thick,
        %fill=red!15,
        inner sep=1.2pt,
        minimum size=8pt,
    }}
    \tikzset{EdgeStyle/.style={thin}}

    % Define vertices with explicit coordinates matching the visual layout
    \Vertex[x=0.0, y=0.0]{1}
    \Vertex[x=-2.7, y=2.0]{2}
    \Vertex[x=-1.2, y=-3.0]{3}
    \Vertex[x=4.0, y=-0.8]{4}
    \Vertex[x=0.0, y=2.7]{5}
    \Vertex[x=1.8, y=-2.4]{6}
    \Vertex[x=2.4, y=1.2]{7}

    % Draw edges connected to the central node 1
    \Edge(1)(2)
    \Edge(1)(3)
    \Edge(1)(4)
    \Edge(1)(5)
    \Edge(1)(6)
    \Edge(1)(7)

    % Draw edges connected to node 7 (excluding node 1)
    \Edge(7)(4)
    \Edge(7)(5)
    \Edge(7)(6)

    % Draw remaining outer edges
    \Edge(2)(5)
    \Edge(3)(6)
    \Edge(4)(6)

\end{tikzpicture}\qquad
\begin{tikzpicture}[scale=0.6]
    \tikzset{VertexStyle/.style={
        shape=circle,
        draw=black,
        %thick,
        %fill=red!15,
        inner sep=1.2pt,
        minimum size=8pt,
    }}
    \tikzset{EdgeStyle/.style={thin}}

    % Define vertices with coordinates matching the visual layout
    \Vertex[x=0.0, y=0.0]{6}
    \Vertex[x=1.2, y=2.6]{16}
    \Vertex[x=3.2, y=1.5]{4}
    \Vertex[x=4.8, y=-0.8]{7}
    \Vertex[x=2.4, y=-2.3]{1}
    \Vertex[x=-1.3, y=-2.8]{3}
    \Vertex[x=-2.7, y=-1.0]{12}

    % Draw edges connected to the central node 6
    \Edge(6)(1)
    \Edge(6)(3)
    \Edge(6)(4)
    \Edge(6)(7)
    \Edge(6)(12)
    \Edge(6)(16)

    % Draw remaining outer edges
    \Edge(1)(3)
    \Edge(1)(4)
    \Edge(1)(7)
    \Edge(3)(12)
    \Edge(4)(7)
    \Edge(4)(16)

\end{tikzpicture}
\caption{Induced subgraphs $\mathcal{G}_{23}(\alpha_1)$ and $\mathcal{G}_{23}(\alpha_6)$.}
\label{fig6}
\end{figure}

%\begin{figure}[h]
%\centering
%\includegraphics[scale=0.5]{./Images/G24i1}\;\;\;
%\includegraphics[scale=0.5]{./Images/G24i6}
%\caption{Induced subgraphs $\mathcal{G}_{24}(\alpha_1)$ and $\mathcal{G}_{24}(\alpha_6)$.}
%\label{fig7}
%\end{figure}

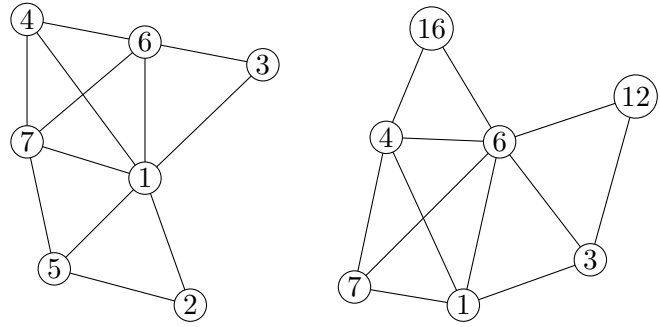
\begin{figure}[h] %G24
\centering
\begin{tikzpicture}[scale=0.6]
    \tikzset{VertexStyle/.style={
        shape=circle,
        draw=black,
        %thick,
        %fill=red!15,
        inner sep=1.2pt,
        minimum size=8pt,
    }}
    \tikzset{EdgeStyle/.style={thin}}

    % Define vertices with explicit coordinates matching the visual layout
    \Vertex[x=0.0, y=0.0]{1}
    \Vertex[x=1.0, y=-2.8]{2}
    \Vertex[x=2.6, y=2.5]{3}
    \Vertex[x=-2.6, y=3.5]{4}
    \Vertex[x=-2.0, y=-2.0]{5}
    \Vertex[x=0.0, y=3.0]{6}
    \Vertex[x=-2.6, y=0.8]{7}

    % Draw edges connected to the central node 1
    \Edge(1)(2)
    \Edge(1)(3)
    \Edge(1)(4)
    \Edge(1)(5)
    \Edge(1)(6)
    \Edge(1)(7)

    % Draw remaining outer edges
    \Edge(2)(5)
    \Edge(3)(6)
    \Edge(4)(6)
    \Edge(4)(7)
    \Edge(5)(7)
    \Edge(6)(7)

\end{tikzpicture}\qquad
\begin{tikzpicture}[scale=0.6]
    \tikzset{VertexStyle/.style={
        shape=circle,
        draw=black,
        %thick,
        %fill=red!15,
        inner sep=1.2pt,
        minimum size=8pt,
    }}
    \tikzset{EdgeStyle/.style={thin}}

    % Define vertices with coordinates matching the visual layout
    \Vertex[x=0.0, y=0.0]{6}
    \Vertex[x=-1.5, y=2.5]{16}
    \Vertex[x=-2.5, y=0.1]{4}
    \Vertex[x=-3.2, y=-3.2]{7}
    \Vertex[x=-0.8, y=-3.6]{1}
    \Vertex[x=2.0, y=-2.6]{3}
    \Vertex[x=3.0, y=1.0]{12}

    % Draw edges connected to the central node 6
    \Edge(6)(1)
    \Edge(6)(3)
    \Edge(6)(4)
    \Edge(6)(7)
    \Edge(6)(12)
    \Edge(6)(16)

    % Draw edges connected to node 4
    \Edge(4)(1)
    \Edge(4)(7)
    \Edge(4)(16)

    % Draw remaining outer edges
    \Edge(1)(3)
    \Edge(1)(7)
    \Edge(3)(12)

\end{tikzpicture}
\caption{Induced subgraphs $\mathcal{G}_{24}(\alpha_1)$ and $\mathcal{G}_{24}(\alpha_6)$.}
\label{fig7}
\end{figure}

%Hence, we are left with two cases: $\mathcal{G}=\mathcal{G}_{35}$ or  %$\mathcal{G}=\mathcal{G}_{40}$. 

\newpage

Hence, we are left with two cases: $\mathcal{G}=\mathcal{G}_{35}$ or  $\mathcal{G}=\mathcal{G}_{40}$. To eliminate the first case, consider the graph shown in Figure~\ref{fig8}.

%Consider the following graph (Fummw in graph6 format):
%Consider the graph shown in Figure~\ref{fig8}.
%\begin{equation*}
%\includegraphics[scale=0.22]{./Images/graph02}
%\end{equation*}

\begin{figure}[h] %G24
\centering
\begin{tikzpicture}[scale=0.6]
    \tikzset{VertexStyle/.style={
        shape=circle,
        draw=black,
        %thick,
        %fill=red!15,
        inner sep=1.2pt,
        minimum size=8pt,
    }}
    \tikzset{EdgeStyle/.style={thin}}

    % Define vertices with coordinates and mathematical labels
    \Vertex[L={$\alpha_{k_1}$}, x=0.0, y=0.0]{1}
    \Vertex[L={$\alpha_{k_2}$}, x=-3.0, y=-2.0]{2}
    \Vertex[L={$\alpha_{k_3}$}, x=-4.0, y=1.5]{3}
    \Vertex[L={$\alpha_{k_4}$}, x=-1.5, y=4.0]{4}
    \Vertex[L={$\alpha_{k_5}$}, x=1.5, y=4.0]{5}
    \Vertex[L={$\alpha_{k_6}$}, x=4.0, y=1.5]{6}
    \Vertex[L={$\alpha_{k_7}$}, x=3.0, y=-2.0]{7}

    % Draw edges connected to the central/hub node k_1
    \Edge(1)(2)
    \Edge(1)(3)
    \Edge(1)(4)
    \Edge(1)(5)
    \Edge(1)(6)
    \Edge(1)(7)

    % Draw the remaining outer edges and crossing segments
    \Edge(2)(3)
    \Edge(2)(4)
    \Edge(3)(4)
    \Edge(3)(5)
    \Edge(4)(5)
    \Edge(4)(6)
    \Edge(5)(6)
    \Edge(5)(7)
    \Edge(6)(7)

\end{tikzpicture}
\caption{Graph Fummw in graph6 format.}
\label{fig8}
\end{figure}
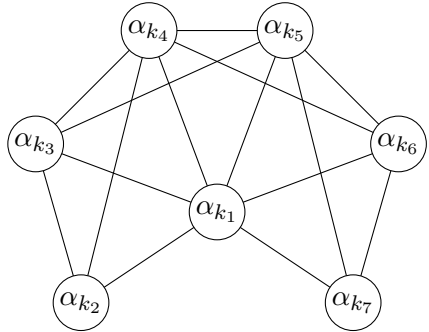

%This graph is an induced subgraph of $\mathcal{G}_{35}$. 
We will prove that this graph cannot be an induced subgraph of $\mathcal{G}$. Assume, on the contrary, that it is. 
By $(ii)$ of Lemma~\ref{lemzstp}, the edge $(\alpha_{k_{1}},\alpha_{k_{2}})$ belongs to 
a zero-sum triangle $\triangle(\alpha_{k_{1}}, \alpha_{k_{2}}, \alpha')$, where $\alpha'\in N(\alpha_{k_1})$. Since $\alpha_{k_{3}}$ and $\alpha_{k_{4}}$ are the only 
vertices in $N(\alpha_{k_{1}})$ that are adjacent to $\alpha_{k_{2}}$, we obtain that 
$\alpha'=\alpha_{k_{3}}$ or $\alpha_{k_{4}}$. Thus, we have to consider two cases: $(i)$ $\triangle(\alpha_{k_{1}}, \alpha_{k_{2}}, \alpha_{k_3})$ is a zero-sum triangle; $(ii)$ $\triangle(\alpha_{k_{1}}, \alpha_{k_{2}}, \alpha_{k_4})$ is a zero-sum triangle.

\underline{Case $(i)$: $\triangle(\alpha_{k_{1}}, \alpha_{k_{2}}, \alpha_{k_3})$ is a zero-sum triangle.}  
By part $(ii)$ of Lemma~\ref{lemzstp}, the edge $(\alpha_{k_{1}},\alpha_{k_{4}})$ belongs to 
a zero-sum triangle $\triangle(\alpha_{k_{1}}, \alpha_{k_{4}}, \alpha'')$, where $\alpha''\in N(\alpha_{k_1})$. 
We have that $\alpha''\in \{\alpha_{k_2},\alpha_{k_3},\alpha_{k_5},\alpha_{k_6}\}$, since $\alpha_{k_4}$ is not adjacent to $\alpha_{k_7}$. Part $(i)$ of Lemma~\ref{lemzstp} implies that neither $\triangle(\alpha_{k_{1}}, \alpha_{k_{4}}, \alpha_{k_{2}})$  
nor $\triangle(\alpha_{k_{1}}, \alpha_{k_{4}},$ $\alpha_{k_{3}})$ is a zero-sum triangle. 
Hence, $\alpha''\in\{\alpha_{k_{5}},\alpha_{k_{6}}\}$.

Assume that $\alpha''=\alpha_{k_{5}}$. Then we have zero-sum triangles  
$\triangle(\alpha_{k_{1}}, \alpha_{k_{4}}, \alpha_{k_{5}})$,  
$\triangle(\alpha_{k_{1}}, \alpha_{k_{2}}, \alpha_{k_3})$, and 
$(\alpha_{k_{2}},\alpha_{k_{4}})$ together with $(\alpha_{k_{3}},\alpha_{k_{5}})$ are 
edges in $\mathcal{G}$. This contradicts part $(iv)$ of Lemma~\ref{lemzstp}.

Assume that $\alpha''=\alpha_{k_{6}}$, so that $\triangle(\alpha_{k_{1}}, \alpha_{k_{4}}, \alpha_{k_{6}})$ is a zero-sum triangle. 
Part $(i)$ of Lemma~\ref{lemzstp} implies that $\triangle(\alpha_{k_{1}}, \alpha_{k_{6}}, \alpha_{k_{7}})$ 
is not a zero-sum triangle, since it shares an edge with the zero-sum triangle $\triangle(\alpha_{k_{1}}, \alpha_{k_{4}}, \alpha_{k_{6}})$. 
Hence, $\triangle(\alpha_{k_{1}}, \alpha_{k_{5}}, \alpha_{k_{7}})$ is a zero-sum triangle. 
Now, we have two zero-sum triangles $\triangle(\alpha_{k_{1}}, \alpha_{k_{4}}, \alpha_{k_{6}})$ and  $\triangle(\alpha_{k_{1}}, \alpha_{k_{5}}, \alpha_{k_{7}})$, and $(\alpha_{k_{4}},\alpha_{k_{5}})$ together with $(\alpha_{k_{6}},\alpha_{k_{7}})$ are 
edges in $\mathcal{G}$. This contradicts part $(iv)$ of Lemma~\ref{lemzstp}.

% Similarly to the previous case $\alpha''=\alpha_{k_{5}}$ we obtain zero-sum triangles  
% $\triangle(\alpha_{k_{4}}, \alpha_{k_{5}}, \tilde{\alpha})$ and 
% $\triangle(\alpha_{k_{6}}, \alpha_{k_{7}}, \alpha''')$ such that 
% $\tilde{\alpha},\alpha'''\notin N[\alpha_{k_{1}}]$. All obtained zero-sum triangles imply the relations
%  \begin{equation}\label{eq6}
% \begin{split}
% %\alpha_{k_{1}} &+ \alpha_{k_{2}} + \alpha_{k_{3}} = 0,\\
% \alpha_{k_{1}} &+ \alpha_{k_{4}} + \alpha_{k_{6}} = 0,\\
% \alpha_{k_{1}} &+ \alpha_{k_{5}} + \alpha_{k_{7}} = 0,\\
% \alpha_{k_{4}} &+ \alpha_{k_{5}} + \tilde{\alpha} = 0,\\
% \alpha_{k_{6}} &+ \alpha_{k_{7}} + \alpha''' = 0.
% \end{split}
% \end{equation}
% By adding the first two relations in \eqref{eq6} and using the expressions for $\tilde{\alpha}$ and $\alpha'''$ 
% from the last two relations in \eqref{eq6}, we obtain
% \begin{equation*}
% 2\alpha_{k_{1}}+ (\alpha_{k_{4}}+\alpha_{k_{5}})+(\alpha_{k_{6}}+\alpha_{k_{7}})=2\alpha_{k_{1}}-\tilde{\alpha}-\alpha'''=0.
% \end{equation*}
% If $\tilde{\alpha}=\alpha'''$, then the last equality implies $\alpha_{k_{1}}=\tilde{\alpha}$, 
% which is impossible, since $\tilde{\alpha}\notin N[\alpha_{k_1}]$. Hence, the set 
% $\{\alpha_{k_1},\tilde{\alpha},\alpha'''\}$ contains at least two distinct numbers. 
% But then the relation $2\alpha_{k_{1}}-\tilde{\alpha}-\alpha'''=0$ contradicts Lemma~\ref{intro7}.
\vspace{0,3cm}

\underline{Case $(ii)$: $\triangle(\alpha_{k_{1}}, \alpha_{k_{2}}, \alpha_{k_4})$ is a zero-sum triangle.} 
Part $(i)$ of Lemma~\ref{lemzstp} implies that neither $\triangle(\alpha_{k_{1}}, \alpha_{k_{3}}, \alpha_{k_{2}})$ 
nor $\triangle(\alpha_{k_{1}}, \alpha_{k_{3}}, \alpha_{k_{4}})$
is a zero-sum triangle, since they share an edge with the zero-sum triangle 
$\triangle(\alpha_{k_{1}}, \alpha_{k_{2}}, \alpha_{k_{4}})$. 
Hence, $\triangle(\alpha_{k_{1}}, \alpha_{k_{3}}, \alpha_{k_{5}})$ is a zero-sum triangle. 
Now, we have two zero-sum triangles $\triangle(\alpha_{k_{1}}, \alpha_{k_{2}}, \alpha_{k_{4}})$ and  
$\triangle(\alpha_{k_{1}}, \alpha_{k_{3}}, \alpha_{k_{5}})$, and $(\alpha_{k_{2}},\alpha_{k_{3}})$ together with $(\alpha_{k_{4}},\alpha_{k_{5}})$ are 
edges in $\mathcal{G}$. This contradicts part $(iv)$ of Lemma~\ref{lemzstp}.

This completes the proof that the graph Fummw cannot be an induced subgraph of $\mathcal{G}$. 
We check with SageMath \cite{sagemath} that the graph $\mathcal{G}_{35}$ contains an induced subgraph isomorphic to Fummw. 
Therefore $\mathcal{G}\neq \mathcal{G}_{35}$ and we are left to consider the last option $\mathcal{G}= \mathcal{G}_{40}$. To do so, consider the graph shown in Figure~\ref{fig9}.
%\vspace{0,3cm}

%Now, consider the following graph (G\textasciitilde{}z$\setminus$z\{ in graph6 format):
%Now, consider the graph shown in Figure~\ref{fig9}.
%\begin{equation*}
%\includegraphics[scale=0.3]{./Images/graph03}
%\end{equation*}

\begin{figure}[h] %G24
\centering
\begin{tikzpicture}[scale=0.6]
    % Define the style for the solid black dots
    \tikzset{VertexStyle/.style={
        shape=circle,
        fill=black,
        inner sep=1.2pt, 
        minimum size=5pt
    }}
    \tikzset{EdgeStyle/.style={thin}}

    % Position the mathematical labels outside the vertices
    \SetVertexLabelOut

    % Define vertices with coordinates forming a regular octagon
    % Lpos controls the precise angle of the external label
    \Vertex[L={$\alpha_{k_1}$}, x=0.0, y=3.5, Lpos=90]{1}
    \Vertex[L={$\alpha_{k_2}$}, x=-2.5, y=2.5, Lpos=135]{2}
    \Vertex[L={$\alpha_{k_3}$}, x=2.5, y=2.5, Lpos=45]{3}
    \Vertex[L={$\alpha_{k_4}$}, x=-3.5, y=0.0, Lpos=180]{4}
    \Vertex[L={$\alpha_{k_5}$}, x=3.5, y=0.0, Lpos=0]{5}
    \Vertex[L={$\alpha_{k_6}$}, x=-2.5, y=-2.5, Lpos=-135]{6}
    \Vertex[L={$\alpha_{k_7}$}, x=2.5, y=-2.5, Lpos=-45]{7}
    \Vertex[L={$\alpha_{k_8}$}, x=0.0, y=-3.5, Lpos=-90]{8}

    % Draw edges from k_1
    \Edge(1)(2) \Edge(1)(3) \Edge(1)(4) \Edge(1)(5) \Edge(1)(6) \Edge(1)(7)
    
    % Draw edges from k_2 (excluding 1)
    \Edge(2)(3) \Edge(2)(4) \Edge(2)(5) \Edge(2)(6) \Edge(2)(8)
    
    % Draw edges from k_3 (excluding 1, 2)
    \Edge(3)(4) \Edge(3)(5) \Edge(3)(7) \Edge(3)(8)
    
    % Draw edges from k_4 (excluding 1, 2, 3)
    \Edge(4)(6) \Edge(4)(7) \Edge(4)(8)
    
    % Draw edges from k_5 (excluding 1, 2, 3, 4)
    \Edge(5)(6) \Edge(5)(7) \Edge(5)(8)
    
    % Draw edges from k_6 (excluding 1 to 5)
    \Edge(6)(7) \Edge(6)(8)
    
    % Draw edges from k_7 (excluding 1 to 6)
    \Edge(7)(8)

\end{tikzpicture}
\caption{Graph G\textasciitilde{}z$\setminus$z\{ in graph6 format.}
\label{fig9}
\end{figure}
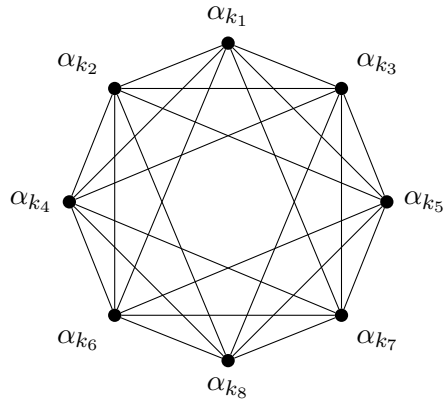

%This graph is an induced subgraph of $\mathcal{G}_{40}$. 
We will prove that this graph cannot be an induced subgraph of $\mathcal{G}$. Assume, to the contrary, that it is. 
Part $(iii)$ of Lemma~\ref{lemrc} implies that there exist three distinct relations 
$\alpha_{k_{1}} + \alpha'_i + \alpha''_i = 0$, $i\in\{1,2,3\}$, which correspond to three distinct zero-sum 
triangles $\triangle(\alpha_{k_{1}}, \alpha'_i, \alpha''_i)$. Then part $(ii)$ of Lemma~\ref{lemzstp} yields 
$\alpha'_i, \alpha''_i\in N(\alpha_{k_{1}})$ for every $i\in\{1,2,3\}$. 
Moreover, applying part $(i)$ of Lemma~\ref{lemzstp}, we obtain that any two of the three sets $\{\alpha'_i, \alpha''_i\}$,  
$i\in\{1,2,3\}$, are disjoint. Hence, 
\[
\{\alpha'_1, \alpha''_1,\alpha'_2, \alpha''_2,\alpha'_3, \alpha''_3\} = N(\alpha_{k_{1}}) = 
\{\alpha_{k_{2}},\alpha_{k_{3}},\alpha_{k_{4}},\alpha_{k_{5}},\alpha_{k_{6}},\alpha_{k_{7}}\}.
\]
Now, adding up the relations $\alpha_{k_{1}} + \alpha'_i + \alpha''_i = 0$, $i\in\{1,2,3\}$, we obtain
\begin{equation}\label{eq8}
3\alpha_{k_{1}}+\alpha_{k_{2}}+\alpha_{k_{3}}+\alpha_{k_{4}}+\alpha_{k_{5}}+\alpha_{k_{6}}+\alpha_{k_{7}}=0.
\end{equation}
Analogously, considering the three relations $\alpha_{k_{8}} + \tilde{\alpha}'_i + \tilde{\alpha}''_i = 0$, $i\in\{1,2,3\}$, 
we obtain
\begin{equation*}
3\alpha_{k_{8}}+\alpha_{k_{2}}+\alpha_{k_{3}}+\alpha_{k_{4}}+\alpha_{k_{5}}+\alpha_{k_{6}}+\alpha_{k_{7}}=0.
\end{equation*}
This equality, together with \eqref{eq8}, imply that $\alpha_{k_{1}}=\alpha_{k_{8}}$, which is a contradiction. 
Hence, the graph G\textasciitilde{}z$\setminus$z\{ cannot be an induced subgraph of $\mathcal{G}$. 
We check with SageMath \cite{sagemath} that the graph $\mathcal{G}_{40}$ contains an induced subgraph isomorphic to G\textasciitilde{}z$\setminus$z\{. 
Therefore $\mathcal{G}\neq \mathcal{G}_{40}$.

Finally, we have proved that $\mathcal{G}\neq\mathcal{G}_{i}$, for $i=1,2,\dots 40$. 
Thus, no algebraic number of degree $d=16$ satisfies the relation \eqref{eq1}.
\end{proof}

\printbibliography

\appendix
\section{SageMath code}

\begin{lstlisting}[language=Python,basicstyle=\ttfamily\scriptsize,caption=SageMath code used in the proof of Theorem~\ref{t1}]

# Load the codes of all 40 vertex-transitive graphs on 16 vertices of degree 6. These graphs are contained in the file 'alltrans16_k=06.txt' which is available at https://zenodo.org/records/4010122
Graph_strings = []
with open('alltrans16_k=06.txt', 'r') as file:
    for line in file:
        Graph_strings.append(line.strip())

# Determine all graphs in Graph_strings that have a subgraph isomorphic to FyTAG (code in graph6 format). 
FyTAG = Graph({1:[2,3,4,5,6,7],2:[3],3:[],4:[5],5:[],6:[7],7:[]})
Graph_FyTAG =[]
for i in range(len(Graph_strings)):
    G = Graph(Graph_strings[i])
    if G.subgraph_search(FyTAG, induced=False) is not None:
        Graph_FyTAG.append(G)

# Determine all graphs in Graph_FyTAG that don't have an induced subgraph isomorphic to FTJG (code in graph6 format). 
FTJG = Graph({1:[2,3,4,5,6,7],2:[3],3:[4],4:[5],5:[6],6:[7],7:[2]})
Graph_FTJG =[]
for G in Graph_FyTAG:
    if G.subgraph_search(FTJG, induced=True) is None:
        Graph_FTJG.append(G)

# Let G be a graph and v -- one of its vertices. Then the induced graph G(v) can be obtained by calling G.subgraph([v]+G.neighbors(v)).
# Example with G=Graph_FTJG nad v=1:
Graph_FTJG[2].subgraph([1]+Graph_FTJG[2].neighbors(1))
\end{lstlisting}

\end{document}